\documentclass[12pt,leqno,a4paper]{amsart}
\usepackage[ansinew]{inputenc}

\usepackage{amsmath}
\usepackage{amssymb}
\usepackage{mathrsfs}
\usepackage{amsthm}
\usepackage[fleqn,tbtags]{mathtools}

\theoremstyle{plain}
\newtheorem{proclaim}{Theorem}[section]
\newtheorem{lemma}[proclaim]{Lemma}

\theoremstyle{definition}

\theoremstyle{remark}
\newtheorem{remark}[proclaim]{Remark}
\numberwithin{equation}{section}
\newtheoremstyle{component}{}{}{}{}{\itshape}{.}{.5em}{\thmnote{#3}#1}
    \theoremstyle{component}

\newcommand{\overbar}[1]{\mkern 1.5mu\overline{\mkern-1.5mu#1\mkern-1.5mu}\mkern 1.5mu}

\newcommand{\I}{\textrm{I}}
\newcommand{\II}{\textrm{II}}

\newcommand{\dif}{\mathrm{d}}

\newcommand{\N}{{\bf N}}

\newcommand{\R}{{\bf R}}

\newcommand{\E}{{\bf E}}

\DeclareMathOperator{\Lip}{Lip}
\DeclareMathOperator{\dist}{dist}

\DeclareMathOperator{\spt}{spt}
\DeclareMathOperator{\dive}{div}

\newcommand\numberthis{\addtocounter{equation}{1}\tag{\theequation}}

\renewcommand{\epsilon}{\varepsilon}

\newcommand{\si}{\mathcal{S}_I}

\newcommand{\sii}{\mathcal{S}_{\II}}

\newcommand{\ex}{\E^{x_0}_{\si, \sii}}
\newcommand{\exx}{\E^{x}_{\si, \sii}}

\newcommand{\ui}{u_{ \mathrm{I} }}
\newcommand{\uii}{u_{ \mathrm{II} }}

\usepackage[pdftex]{graphicx}
\title[Dynamic programming principle related to the $p$-Laplacian]{A dynamic programming principle with continuous solutions related to the $p$-Laplacian, $1 < p < \infty$}
\author{Hans Hartikainen}
\address{University of Jyvaskyla, Department of Mathematics and Statistics, P.O. Box 35, FI-40014 University of Jyvaskyla, Finland. }
\email{hans.k.k.hartikainen@jyu.fi}
\begin{document}
\begin{abstract}
We study a Dynamic Programming Principle related to the $p$-Laplacian for $1 < p < \infty.$ The main results are existence, uniqueness and continuity of solutions.
\end{abstract}
\maketitle

\section{Introduction}
Game-theoretic methods have recently emerged as a novel approach to nonlinear partial differential equations. In particular, tug-of-war games related to $p$-Laplace-type equations have garnered significant attention. The groundwork was laid in the seminal papers \cite{peresinfty, peres}. Game-theoretic views have since led to simplified proof of theorems concerning $p$- and $\infty$-harmonic functions, as in \cite{asinfty, harnack}, as well as stimulated plenty of other related research, as in \cite{ant, arr, caf, pezzorossi, gomezrossi, lewimanf, liuschikorra, mikko, manparros, manrossom, perpetsom, ruo}.

One way to establish the link between a PDE and the game to show that there exists a value function for the game, satisfying a Dynamic Programming Principle (DPP for short). This DPP can then be interpreted as a discrete approximation to the PDE in question. A solution of the PDE can be acquired by taking a suitable limit of solutions of the DPP.   

In this paper, inspired by \cite{kmp, peres}, we propose an orthogonal noise tug-of-war variant related to the $p$-Laplace equation
\[
\Delta_p u = \dive(|\nabla u|^{p-2} \nabla u) = 0
\]
 for $1 < p < \infty$. The related DPP is
\begin{align*}
&u(x) = \frac{1 - \delta(x)}{2} \left[ \sup_{v \in\partial B_\epsilon}\left( \alpha u(x+v)   + \beta \int \! u(x+z) \,\dif \mu_v (z)  \right)  \right. \numberthis \label{DPP} \\& + \left. 
\inf_{v \in \partial B_\epsilon}\left(\alpha u(x+v)   + \beta \int \! u(x+z) \,\dif \mu_v (z)\right) \right] +\delta(x)F(x), 
\end{align*}
where $\epsilon > 0$, $\delta$ is a boundary cutoff function, $F$ the boundary data on the thickened boundary and $\mu_v$ is the uniform $(n-1)$-dimensional probability measure over the $(n-1)$-dimensional closed disk of radius $\epsilon$ orthogonal to the vector $v$. The constants $\alpha, \beta \geq 0$ are determined by $p$ and the dimension $n$ as
\[
\alpha = \frac{p-1}{p+n} \quad \beta = \frac{n+1}{p+n}.
\]
Our results for the DPP also hold in the case $\alpha = 0$. However, the corresponding PDE for this case is not the $1$-Laplace equation.

The boundary correction term allows us to sidestep some subtle measurability and continuity issues typically arising in similar cases. To see the connection between the DPP and the $p$-Laplacian, we refer to \cite{kmp}, where analogous computations were carried out.

For a bounded domain and Lipschitz continuous boundary data, we prove the existence, uniqueness and continuity of solutions of the DPP (\ref{DPP}). To the best of our knowledge, these types of results are original in the range $1 < p < 2$.

In the first section, we use analytic tools to prove an existence of both a lower and an upper semicontinuous solution to the DPP.  We construct a sequence of functions iterating under an operator closely related to the DPP.  We then prove and utilize suitable monotonicity and continuity-preserving properties for the operator to establish favorable properties for the sequence. Finally, we deduce semicontinuity-type properties for the operator to guarantee that the limit of the iteration process satisfies the DPP.

We begin the next section by introducing the game. Among other considerations, we show that the game ends almost surely in finite time, irregardless of the choices made by the players. Utilizing the value functions related to the game, we prove that the two semicontinuous solutions to the DPP must agree, implying the existence of continuous solutions.

The final section centers around a maximum principle and its consequences. A brief argument utilizing the results from the first section shows that if a measurable function is bounded between the extremal values of the boundary data, it has to be continuous. We then show that there cannot exist a measurable function not bounded between these extremal values. Therefore, any measurable solution to the DPP is also continuous. Furthermore, combining this with a brief argument gives an uniqueness result.

\section{Existence of a function satisfying the Dynamic Programming Principle}
In this section, we show the existence of semicontinuous solutions for the DPP (\ref{DPP})

\begin{align*}
&u(x) = \frac{1 - \delta(x)}{2} \left[ \sup_{v \in\partial B_\epsilon}\left( \alpha u(x+v)   + \beta \int \! u(x+z) \,\dif \mu_v (z)  \right)  \right.  \\& + \left. 
\inf_{v \in\partial B_\epsilon}\left(\alpha u(x+v)   + \beta \int \! u(x+z) \,\dif \mu_v (z)\right) \right] +\delta(x)F(x), 
\end{align*}

Let $\Omega$ be a bounded domain in $\R^n$. We denote the outside boundary strip of width $\epsilon$ as
\[
O_\epsilon := \{ x \in \R^n \setminus \Omega \, | \, \dist(x, \partial \Omega) \leq \epsilon \},
\]
the inside boundary strip of width $\epsilon$ as
\[
I_\epsilon := \{ x \in \Omega \, | \, \dist(x, \partial \Omega) \leq \epsilon \}
\]
and the thickened domain
\[
\Omega_\epsilon = \Omega \cup O_\epsilon.
\]
We define a cutoff function in $\Omega_\epsilon$ by setting
\[
\delta(x) := \begin{cases} 0 & \mbox{if }x \in  \Omega \setminus I_\epsilon \\ 1-\epsilon^{-1}\dist(x,\partial\Omega) & \mbox{if } x \in I_\epsilon \\ 1 & \mbox{if } x \in O_\epsilon.\end{cases}
\]
Let $\alpha, \beta > 0$, $\alpha + \beta = 1$.
We define an operator $\tilde{I}$ acting on measurable functions as
\begin{align*}
\tilde{I}u(x) =& \frac{1}{2}\left[ \sup_{v \in\partial B_\epsilon}\left( \alpha u(x+v) + \beta \int \! u(x+z) \,\dif \mu_v (z)  \right)  \right. \\&  + \left. 
\inf_{v \in\partial B_\epsilon}\left( \alpha u(x+v) + \beta \int \! u(x+z) \,\dif \mu_v (z)\right) \right]
\end{align*}
and the boundary-corrected version
\begin{align*}
Iu(x) =& \frac{1 - \delta(x)}{2} \left[ \sup_{v \in\partial B_\epsilon}\left( \alpha u(x+v) + \beta \int \! u(x+z) \,\dif \mu_v (z)  \right)   \right. \\& + \left. 
\inf_{v \in\partial B_\epsilon}\left( \alpha u(x+v) + \beta \int \! u(x+z) \,\dif \mu_v (z)  \right)  \right] + \delta(x)F(x), \numberthis \label{tildei}
\end{align*}
where $\mu_v$ is the uniform $(n-1)$-dimensional probability measure over the $(n-1)$-dimensional closed disk of radius $\epsilon$ orthogonal to the vector $v$. The boundary data $F$ is a real-valued Lipschitz continuous function defined on the boundary strip
\[
\Gamma_\epsilon = \{ x \in \Omega_\epsilon \, | \, \dist(x, \partial \Omega) \leq \epsilon \}.
\]

We now show that the operator $I$ maps Lipschitz functions as to Lipschitz functions. We do this in multiple steps.
\begin{lemma}\label{lipint}
For any Lipschitz function $u: \R^n \to \R$
\[
\left| \int \! u(x+z) \,\dif \mu_v (z) - \int \! u(y+z) \, \dif \mu_v(z) \right| \leq \Lip_u |x - y|,
\]
where
\[
\Lip_{u}(x) := \sup_{x \in \Omega_\epsilon}\limsup_{y \to x} \frac{|u(y)-u(x)|}{|y-x|}.
\]
\end{lemma}
\begin{proof}
The proof is a direct calculation. Namely
\begin{align*}
&\left| \int \! u(x+z) \,\dif \mu_v (z) - \int \! u(y+z) \, \dif \mu_v(z) \right| \\
 & = \left| \int \! u(x+z) -  u(y+z) \, \dif \mu_v(z) \right| \\
& \leq  \int \! \left |u(x+z) -  u(y+z) \right| \, \dif \mu_v(z) \\
& \leq  \int \! \Lip_u \left| x+z-(y-z) \right| \, \dif \mu_v(z) \\
& =  \Lip_u |x-y|.
\end{align*}
For the last step we used the fact that the measure $\mu_v$ is a probability measure.
\end{proof}

We now show the Lipschitz continuity for the supremum and infimum terms.
\begin{lemma}\label{limsupvetalemma}
Assume that $u: \Omega_\epsilon \to \R$ is a Lipschitz function. Then the function 
\[
\overline{W}(x) := \sup_{v \in\partial B_\epsilon}\left[  \alpha u(x+v) + \beta \int \! u(x+z) \,\dif \mu_v (z) \right]
\]
is also a Lipschitz function (defined in $\Omega$). Furthermore, $Lip_{\overline{W}} \leq 3 Lip_u$.
\end{lemma}
\begin{proof}

It suffices to show that there exists a uniform bound for the pointwise Lipschitz constant given by
\[
\Lip_{\overline{W}}(x) := \limsup_{y \to x} \frac{|\overline{W}(y)-\overline{W}(x)|}{|y-x|} \leq  3\Lip_u.
\]
Denote
\begin{equation*}
W(x, v) := \alpha u(x + v) + \beta \int \! u(x+z) \,\dif \mu_{v} (z).
\end{equation*}
Let $\eta > 0$ be arbitrary and choose a vector $v_\eta \in\partial B_\epsilon$ such that
\[
\sup_{v \in\partial B_\epsilon} W(x,v) - W(x, v_\eta) \leq \eta.
\]
Then we must also have  
\begin{equation} \label{limsupveta}
\begin{aligned}
& \sup_{v \in\partial B_\epsilon}W(y, v) - W(y, v_\eta)  \\
\leq & \sup_{v \in\partial B_\epsilon}W(x, v) - W(x, v_\eta) + 2\Lip_u |x-y| \\
\leq &  2\Lip_u |x-y| + \eta,
\end{aligned}
\end{equation}
using the Lemma \ref{lipint} twice for the first inequality.

The Lipschitz estimate can now be acquired as follows
\begin{align*}
& |\overbar{W}(y)-\overbar{W}(x)| \\
\leq & | W(y, v_\eta) - W(x, v_\eta)| + 2\Lip_u |x-y| + 2\eta \\
\leq & 3\Lip_u |x-y| + 2\eta.
\end{align*}
For the first inequality, we used $(\ref{limsupveta})$ and the choice of $v_\eta$, the second inequality utilizes $(\ref{lipint})$. Diving with $|x-y|$ and taking lim\,sup finishes the proof.
\end{proof}
By considering $-u$ in the previous proof, we immediately get the analogous claim for
\[
\underline{W}(x) := \inf_{v \in\partial B_\epsilon}\left[ \alpha \int \! u(x+z) \,\dif \mu_v (z) + \beta u(x+v) \right].
\]

\begin{lemma}\label{lipcont}
If $u \in \Lip(\Omega_\epsilon)$, then  $Iu \in \Lip(\Omega_\epsilon)$.
\end{lemma}
\begin{proof}
Lipschitz continuity in the outer strip $O_\epsilon$ immediately follows, since
\[
(Iu)|_{O_\epsilon} = F|_{O_\epsilon}.
\]
Next, we show Lipschitz continuity in $\Omega$.
Using the fact
\[
\Lip_{fg} \leq |f|\Lip_{g} + |g|\Lip_{f}
\]
we first note that
\[
\Lip_{\delta F}(x) \leq \delta(x) \Lip_F + \epsilon^{-1}||F||_\infty < \infty.
\]
Using
\[
\Lip_{af+bg} \leq |a|\Lip_f + |b|\Lip_g,
\]
Lemma \ref{limsupvetalemma} and the fact \[\tilde{I}u(x) = \frac{1}{2}(\underline{W}(x) + \overline{W}(x)),\]
we have $\Lip_{\tilde{I}u} \leq 3 \Lip_{u}$. Putting these facts together gives
\begin{align*}
\Lip_{Iu} &\leq \Lip_{\tilde{I}u(1-\delta)} + \Lip_{\delta F} \\
&\leq (1-\delta(x))\Lip_{\tilde{I}u} + \epsilon^{-1}||u||_\infty + \delta(x)\Lip_F + \epsilon^{-1}||F||_\infty \\
&\leq \max(3 \Lip_{u}, \Lip_F) + \epsilon^{-1}(||F||_\infty + ||u||_\infty)< \infty,
\end{align*}
implying Lipschitz continuity in $\Omega$.

Furthermore, we have for any $\tilde{x} \in \partial \Omega$
\begin{align*}
\lim_{x \to \tilde{x}} Iu(x) &= \lim_{x \to \tilde{x}} \frac{1 - \delta(x)}{2} \left[ \sup_{v \in\partial B_\epsilon}\left( \alpha u(x+v) + \beta \int \! u(x+z) \,\dif \mu_v (z)  \right)   \right. \\& + \left. 
\inf_{v \in\partial B_\epsilon}\left( \alpha u(x+v) + \beta \int \! u(x+z) \,\dif \mu_v (z)  \right)  \right] + \delta(x)F(x) \\
&= \lim_{x \to \tilde{x}} \delta(x)F(x) = F(\tilde{x}).  
\end{align*}
We have thus verified continuity in $\Omega_\epsilon$ and Lipschitz continuity separately in $\Omega_\epsilon$ and $O_\epsilon$. These results together imply Lipschitz continuity in the whole domain $\Omega_\epsilon$.
\end{proof}
We now choose $u_0(x) \equiv \inf F$ and define \[ u_n(x) := I^n u_0(x). \] The operator $I$ is easily seen to be monotone, i.e. if $u \leq v$, then $Iu \leq Iv$. Since necessarily $u_0 \leq Iu_0$, the sequence $u_n$ is increasing. Furthermore, the following lemma provides an upper bound.
\begin{lemma}
Let $u_n$ be defined as above. Then \[ u_n \leq \sup F \] for any $n \in \N$.
\end{lemma}
\begin{proof}
The proof is by induction. The base case $u_0 = \inf F \leq \sup F$ is immediate. For the general case, consider
\begin{align*}
u_{n+1}(x) &=  \frac{1 - \delta(x)}{2} \left[ \sup_{v \in\partial B_\epsilon}\left( \alpha u_n(x+v) + \beta\int \! u_n(x+z) \,\dif \mu_v (z)   \right]  \right. \\& + \left. 
\inf_{v \in\partial B_\epsilon}\left(\alpha u_n(x+v) + \beta  \int \! u_n(x+z) \,\dif \mu_v (z) \right) \right] + \delta(x)F(x) \\
&\leq  (1 - \delta(x)) \sup F + \delta(x) \sup F , 
\end{align*}
where we applied the induction hypothesis $u_n \leq \sup F$.
 \end{proof} As a pointwise limit of an increasing and bounded sequence, the function
\begin{equation*}
u(x) = \lim_{n \to \infty} u_n(x)
\end{equation*}
is well defined. Furthermore, Lemma $\ref{lipcont}$ implies that the functions $u_n$ are continuous. An increasing limit of continuous functions is lower semicontinuous, thus in particular measurable.  

In the proof of the main claim of this section, we utilize the following lemma.
\begin{lemma}\label{kompjatk}
The functions
\begin{equation*}
W_n(x,v) :=  \alpha u_n(x+v)  + \beta \int \! u_n(x+z) \,\dif \mu_v (z)  
\end{equation*}
are continuous with respect to the variable $v$ for each fixed $x$. 
\end{lemma}
\begin{proof}
The proof only utilizes the Lipschitz continuity on the functions $u_n$. It suffices to show that the function
\[
h(v) = \int \! u(x+z) \,\dif \mu_v (z)
\]
is continuous with any Lipschitz continuous function $u$ and point $x$.

Let $v_1, v_2 \in \partial B_\epsilon$. Denote
\[
(v_1)^\perp := \{ w \in \R^n \, | \, \langle w , v_1 \rangle = 0\},
\]
and similarily for $(v_2)^\perp$. Given these vectors, we can find a rotation $L: (v_1)^\perp \to (v_2)^\perp$ for which
\[|L(z) - z| \leq C|z||v_1 - v_2|, \] 
where the constant $C > 0$  can be chosen independent of vectors $v_1$ and $v_2$.  Using this fact, we can estimate
\begin{align*}
&|h(v_1) - h(v_2)| \\ = &\left| \int \! u(x+z) \,\dif \mu_{v_1} (z) - \int \! u(x+z) \,\dif \mu_{v_2} (z) \right| \\
\leq &\left| \int \! u(x+L(z))  \,\dif \mu_{v_1} (z) - \int \! u(x+z) \,\dif \mu_{v_2} (z) \right| \\
 &+ \int \! |u(x+z) - u(x+L(z))|  \,\dif \mu_{v_1} (z)  \\
\leq  &\left | \int \! u(x+z) - u(x+z) \, \dif \mu_{v_2}(z)  \right| + \int \! \Lip_u |L(z) - z| \,\dif \mu_{v_1} (z) \\
\leq & \int \! \Lip_u C |z| |v_1 - v_2|  \,\dif \mu_{v_1} (z) \leq \Lip_u C \epsilon |v_1 - v_2|.
\end{align*}
The Lipschitz-continuity of $u$ was utilized in the second equality. For the change of integration measure, we used the fact that the determinant of the Jacobian matrix of a rotation is one. The last inequality is due to the inclusion $\spt \mu_{v_2} \subset \overbar{B_\epsilon}.$
\end{proof}
\begin{remark}\label{remark1}

Note that this implies the Lipschitz continuity of the functions
\begin{align*}
& \underline{u_n}(x,v) := \frac{1-\delta(x)}{2}\left[ \alpha  u_n(x+v) + \beta  \int \! u_n(x+z) \,\dif \mu_v (z)\right] \\
+ &  \frac{1 - \delta(x)}{2} \left[ \inf_{\tilde{v} \in\partial B_\epsilon}\left( \alpha u_n(x+\tilde{v}) + \beta  \int \! u_n(x+z) \,\dif \mu_{\tilde{v}} (z)   \right)  \right] + \delta(x)F(x).
\end{align*}
with respect to both variables.
As before, one can verify that this sequence is increasing, and bounded from above by $\sup F$. Furthermore, 
\[
\sup_{v \in\partial B_\epsilon} \lim_{n \to \infty} \underline{u_n}(x,v)  = u(x).
\]
We will use these facts in the proof of Lemma \ref{uniq}.
\end{remark}
\begin{proclaim}
The function $u := \lim_{n \to \infty} I^n u_0$ satisfies the DPP.
\end{proclaim}
\begin{proof}
We first prove the following auxiliary claims.
\begin{multline}\label{infdpp}
\lim_{n \to \infty} \inf_{v \in\partial B_\epsilon} \left[  \alpha u_n(x+v)  + \beta  \int \! u_n(x + z) \, \dif \mu_v(z) \right] \\ = \inf_{v \in\partial B_\epsilon} \left[ \alpha u(x + v)  + \beta \int \!  u(x + z) \, \dif \mu_v(z) \right].
\end{multline}
\begin{multline}\label{supdpp}
\lim_{n \to \infty} \sup_{v \in\partial B_\epsilon} \left[ \alpha u_n(x+v)  + \beta \int \!  u_n(x + z) \, \dif \mu_v(z)\right] \\ = \sup_{v \in\partial B_\epsilon} \left[ \alpha u(x + v) + \beta \int \!  u(x + z) \, \dif \mu_v(z)  \right].
\end{multline}

We begin with the proof of (\ref{supdpp}). Since the sequence of functions $(u_n)_{n \in \N}$ is increasing, the sequence
\[
a_n := \sup_{v \in\partial B_\epsilon} \left[\alpha u_n(x + v) + \beta  \int \! x u_n(x + z) \, \dif \mu_v(z)\right]
\]
must also be increasing. Note also that the sequence is bounded above by 
\[
a := \sup_{v \in\partial B_\epsilon} \left[ \alpha u(x + v) + \beta   \int \!  u(x + z) \, \dif \mu_v(z) \right].
\]
 For any arbitrary $\eta > 0$, pick $\tilde{v} \in\partial B_\epsilon$  such that
\[
 \alpha u(x + \tilde{v}) + \beta  \int \!  u(x + z) \, \dif \mu_{\tilde{v}}(z) \geq a - \eta
\]
Now, by the dominated convergence theorem (recall that $\inf F \leq u_n \leq \sup F$).
\begin{align*}
&\lim_{n \to \infty} \left[ \alpha u_n(x+\tilde{v}) + \beta  \int \!  u_n(x + z) \, \dif \mu_{\tilde{v}}(z)\right] \\ & =  \alpha u(x + \tilde{v}) + \beta  \int \!  u(x + z) \, \dif \mu_{\tilde{v}}(z)   \\
& \geq a - \eta,
\end{align*}
implying (\ref{supdpp}).

For the proof of (\ref{infdpp}), we use Cantor's lemma. Fix $x \in \Omega$ and denote
\[
W_n(x,v) := \alpha u_n(x+v) + \beta \int \! u_n(x+z) \,\dif \mu_v (z).
\]
By Lemma \ref{kompjatk}, the functions $w_n(x,v)$ are continuous with respect to the variable $v$ for each fixed $x$. Thus, the sets
\[
K^n_b := \{ v \in\partial B_\epsilon \, | \, W_n(x,v) \leq b \}
\]
are closed for every $b \in \R$, therefore as subsets of $\partial B_\epsilon$, also compact. Since the function sequence $(u_n)_{n \in \N}$ is increasing, we have $K^{n+1}_b \subset K^n_b$ for all $n \in \N$. Now choose 
\[
 b = \lim_{n \to \infty} \inf_{v \in\partial B_\epsilon} \left[  \alpha  u_n(x + v)  + \beta \int \!  u_n(x + z) \, \dif \mu_v(z) \right]. 
\] 
With this choice, we immediately have $K^n_b \neq \emptyset$ for all $n \in \N$. The conditions for Cantor's lemma are satisfied, which implies
\[
\bigcap_{n = 1}^\infty K^n_b \neq \emptyset.
\]
Let $\tilde{v} \in \cap_{n = 1}^\infty K^n_b$. Then 
\begin{align*}
b &\leq \inf_{v \in B_ \epsilon} \left[ \alpha  u(x+v) + \beta  \int \! u(x+z) \,\dif \mu_v (z) \right] \\ 
&\leq  \alpha u(x+\tilde{v}) + \beta  \int \! u(x+z) \,\dif \mu_{\tilde{v}} (z)  \\
&=  \left[ \alpha u_n(x+\tilde{v}) + \beta  \int \! u_n(x+z) \,\dif \mu_{\tilde{v}} (z) \right] \\
&\leq b. 
\end{align*}
The first inequality follows from the definition of $b$ and the fact that the sequence  $(u_n)_{n \in \N}$ is increasing. For the first equality, we used the monotone convergence theorem. We have thus shown (\ref{infdpp}).

The proof of the main claim is an application of the auxiliary claims proved above
\begin{align*}
u &= \lim_{n \to \infty} I^n u_0 \\
&=   \frac{1 - \delta(x)}{2} \left[ \lim_{n \to \infty} \sup_{v \in\partial B_\epsilon}\left( \alpha u_{n-1}(x+v)  + \beta \int \! u_{n-1}(x+z) \,\dif \mu_v (z)  \right)  \right. \\& + \left. 
\lim_{n \to \infty} \inf_{v \in\partial B_\epsilon}\left(\alpha u_{n-1}(x+v)  + \beta \int \! u_{n-1}(x+z) \,\dif \mu_v (z) \right) \right] + \delta(x)F(x)   \\
&=   \frac{1 - \delta(x)}{2} \left[ \sup_{v \in\partial B_\epsilon}\left( \alpha  u(x+v)  + \beta \int \! u(x+z) \,\dif \mu_v (z) \right)  \right. \\& + \left. 
\lim_{n \to \infty} \inf_{v \in\partial B_\epsilon}\left(\alpha  u(x+v) + \beta \int \! u(x+z) \,\dif \mu_v (z)  \right) \right] + \delta(x)F(x) . \qedhere
\end{align*}
\end{proof}
As a final result, we have
\begin{proclaim}\label{exis}
There exists a lower and an upper semicontinuous solution to the DPP.
\end{proclaim}
The upper semicontinuous solution is acquired and proved to be a solution by performing an analogous argument with the choice $u_0 = \sup F$. In the sequel, we shall denote the lower semicontinuous solution as $\underline{u}$ and the upper semicontinuous solution as $\overbar{u}$.
\begin{remark}\label{uepis}
Since the iteration operator $I$ is monotone, we have
\[
\underline{u} = \lim_{n \to \infty} I^n \inf F \leq \lim_{n \to \infty} I^n \sup F = \overbar{u}.
\]
\end{remark}
\section{Uniqueness}
In this section we shall prove a uniqueness result for the DPP (\ref{DPP}) using a stochastic two-player tug-of-war game.
Throughout this section we shall assume that $\beta > 0$.

To carry out the uniqueness proof, we utilize a tug-of-war game with noise inspired by \cite{kmp, peres}. 
We denote
\[
 \Gamma_\epsilon := \{ x \in \R^n \, | \, \dist(x, \partial \Omega) \leq \epsilon \}.
\]
For the general setting of the game, suppose we are given $\Omega \subset \R^n$  a domain and $F: \Gamma_\epsilon \to \R$ a Lipschitz function and constants $\epsilon > 0$, $0 \leq \alpha, \beta \leq 1$  $\alpha + \beta = 1$. The game involves two players. At the beginning of the game, a game token is
placed at a point $x_0 \in \Omega$. Both players declare their intended move vectors \[ v_1^\mathrm{I}, v_1^{\mathrm{II}} \in \partial B_\epsilon \] and then a fair coin is tossed to determine which
of these vectors will be used in the game. Let us say that Player I won, and his intended move vector was $v_1^\mathrm{I}$. Then the next location $x_1$ of the game token is
determined as follows. With probability $\alpha$, the game token is moved to the point 
\[ x_1 = x_0 + v_1^\mathrm{I}, \]
and with probability $\beta$, a displacement vector $v_1^*$ is chosen from the $(n-1)$-dimensional closed disk that has radius $\epsilon$ and is orthogonal to the vector $v_1^\mathrm{I}$. Thus
the new location of the game token is
\[ x_1 = x_0 + v_1^*.\]
If Player II wins, the same procedure is carried out for the vector $v_1^{\mathrm{II}}$. Similar turns are then played until the game token reaches the inside boundary strip
\[
I_\epsilon := \{ x \in \Omega \, | \, \dist(x, \Omega) \leq \epsilon \}.
\]
If the game token at turn $n$ is at the location $x_n \in I_\epsilon$, the game ends with probability 
\[
 1-\epsilon^{-1}\dist(x,\partial\Omega).
\]
If the game token ever reaches the outer boundary strip 
\[
O_\epsilon := \{ x \in \R^n \setminus \Omega \, | \, \dist(x, \partial \Omega) \leq \epsilon \},
\]
the game is immediately terminated.

Now denote the final location of the game token as $x_\tau$. The outcome of the game is that Player II then pays Player I the amount prescribed by the quantity $F(x_\tau)$.

The choices of intended move vectors are formalized as strategies, which take the current history of the game $(x_1, \dots, x_n)$ as an argument and maps it to an intended move $v \in \partial B_\epsilon$. For our purposes, the strategies can be viewed as maps
\[
 \si, \sii: \cup_{k=1}^\infty (\Omega_\epsilon)^k \to \partial B_\epsilon.
\]
For measure-theoretic reasons, we will only allow Borel-measurable maps as admissible strategies. Now, given a starting point $x_0 \in \Omega$ and two Borel-measurable strategies $\si$ and $\sii$ we can define a measure on the space of game trajectories (using Kolmogorov's extension theorem), which we shall denote as  $\textbf{P}^{x_0}_{\si, \sii}$. 
This can be then used to compute the expected outcome for the game under these conditions, namely $\ex[F(x_\tau)]$.

For a general tug-of-war game defined along similar lines, it is not always clear that the point $x_\tau$ is well defined, since the game may stay in the interior of the domain indefinitely. However, for this variant, the following lemma
shows that $x_\tau$ exists almost surely when $\beta \neq 0$.
\begin{lemma}\label{gameends}
Let $\beta > 0$. If the game domain $\Omega$ is bounded, the game will almost surely end in a finite time, irregardless of the strategies $\si$ and $\sii$.
\end{lemma}
\begin{proof}
Without loss of generality, we can assume that the game domain is a ball with radius $r > 0$, centered at the starting point $x_0 = 0$.
We first prove that the game ends in the case $\beta = 1$, and then show how this result implies the general case.

We aim to find a sequence of events of positive probability, such that the modulus of the game location vectors $x_n$ will always grow in a suitable fashion. 

Denote the vectors used in the game as $v_1, v_2, v_3, \dots$ For all the random displacements $w_i \perp v_i$,  $i \in \N$, we shall restrict ourselves to cases in which $|w_i| \geq \epsilon/2,$ since the probability of this happening on a given single turn is positive. 

For $j \geq 0$, we have
\begin{align*}
|x_{j+1}|^2 &= |x_j + w_j|^2  \\
&= \langle x_j + w_j , x_j + w_j \rangle \\
&= |x_j|^2 + |w_j|^2 + 2 \langle x_j, w_j \rangle.
\end{align*}
With at least probability $\frac{1}{2}$, the term $\langle x_j, w_j \rangle$ will be positive. Thus for all $j \geq 0$ we have a positive probability of choosing $w_j$ for which
\[
|x_{j+1}|^2  \geq |x_j|^2+ |w_j|^2 \geq |x_j|^2 + \epsilon^2/4.
\]
Making $\tilde{j} = \lceil 16r^2/\epsilon^2 \rceil $ such consecutive choices results in
\[
|x_{\tilde{j}}|^2 \geq \tilde{j}\epsilon^2/4  \geq 4 r^2,
\]
thus exiting the ball $B_{2r} (0)$, thus in particular the ball $B_{r}(0)$. We have thus established a lower bound on the probability $\theta > 0$ that the game token exits $B_{r}(0)$ before or at time $\tilde{j}$. 

By shifting the starting time to $\tilde{t}$ and letting $x_{\tilde{t}} \in B_{r}(0)$, we have as a corollary that the probability of the game token exiting the ball $B_{2r}(x_{\tilde{t}}) \supset B_{r}(0)$ during turns $\tilde{t} + 1, \tilde{t} + 2, \cdots, \tilde{t} + \tilde{j}$ is at least $\theta$. This establishes an upper bound on the probability that the game hasn't ended before time $k\tilde{j}$ as $(1-\theta)^k$ for any $k \in \N$. Taking $k$ to infinity yields our claim in the case $\beta = 1$.
 
For the general case, consider the following: For any positive integer $j$, an infinite sequence of $(\alpha, \beta)$-biased coin tosses will contain infinitely many $\alpha$-streaks of length $j$. Thus, it suffices to show that during a streak of a sufficient length, in this case $\lceil 16r^2/\epsilon^2 \rceil $, there is a positive probability that the game ends. We can now apply similar reasoning as given above to conclude the proof.
\end{proof}

Now that we have established the well-foundedness of the game, we can discuss the connection to the Dynamic Programming Principle (\ref{DPP}). 

Assume that both players are skilled at the game. Player II offers to play this game starting from point $x$ for some price $u(x)$. How much should Player I be willing to pay and how much should Player II charge? Depending on the point of view, there are two possible answers. For Player I, the answer is 
\[
u_I(x) = \sup_{\si} \inf_{\sii} \exx[F(x_\tau)]
\]
and for Player II
\[
u_{II}(x) = \inf_{\sii} \sup_{\si} \exx[F(x_\tau)].
\]
Using only the definitions of inf and sup, one can readily verify that \begin{equation}\label{peliepis}\ui \leq \uii. \end{equation} 
For our game, we will have $\ui = \uii$, which we will prove along Theorem \ref{uniq}.

Our aim is to relate the game to the DPP (\ref{DPP}) heuristically speaking as follows: Starting a game from point $x$ and evaluating all the possible outcomes for the next location of the game token, we should recieve the equation
\begin{align*}
& \ui(x) =   \frac{1 - \delta(x)}{2} \left[ \sup_{v \in\partial B_\epsilon}\left( \alpha \ui(x+v)  + \beta  \int \! \ui(x+z) \,\dif \mu_v (z) \right)  \right. \\& + \left. 
\inf_{v \in\partial B_\epsilon}\left(\alpha \ui(x+v) + \beta \int \! \ui(x+z) \,\dif \mu_v (z)  \right) \right] + \delta(x)F(x),
\end{align*} i.e. $\ui$ satisfies the DPP given in (\ref{DPP}). A similar claim should also hold for $\uii.$  We prove this as a by-product of Theorem \ref{uniq}.  

For the uniqueness proof, we will need to assign a strategy, which we will need to check to be Borel. To this end, we utilize the following special case of a measurable selection theorem due to Sch\"al, see \cite{schal1} and \cite{schal2}.
\begin{proclaim}\label{schal}
Let $T$ be a topological space and $Y$ a compact metric space. Furthermore, let $v: T \times Y \to \R$ be a pointwise limit of a non-increasing sequence $(v_n)$, such that for each $t \in T$ and $n \in \N$, the function $v_n(t,\cdot)$ is continuous. Then there exists a Borel-measurable function $g: T \to Y$ such that
\[
v(t, g(t)) = \sup_{y \in Y}v(t,y)
\]
for every $t \in T$.
\end{proclaim}

For a more general formulation and proof, see for example \cite[Theorem 5.3.1]{srivastava}. The version for the limit of a non-decreasing sequence is an immediate corollary. We shall use this claim for the functions 
\begin{align*}
& \underline{u_n}(x,v) := \frac{1-\delta(x)}{2}\left[ \alpha  u_n(x+v) + \beta  \int \! u_n(x+z) \,\dif \mu_v (z)\right] \\
+ &  \frac{1 - \delta(x)}{2} \left[ \inf_{\tilde{v} \in\partial B_\epsilon}\left( \alpha u_n(x+\tilde{v}) + \beta  \int \! u_n(x+z) \,\dif \mu_{\tilde{v}} (z)   \right)  \right] + \delta(x)F(x).
\end{align*} defined in Remark \ref{remark1}, where also the prerequisites for using Theorem \ref{schal} were shown.
\begin{proclaim}\label{uniq}
For the function $\underline{u}$ given in Theorem \ref{exis}, we have $\uii  \leq \underline{u}$.
\end{proclaim}
\begin{proof}

We prescribe the Borel-measurable strategy $\sii^*$ given by Theorem \ref{schal} to Player II: At point $x_k$, a vector $v$ is chosen such that
\begin{align*}
&  \alpha  \underline{u}(x_k+v) + \beta  \int \! \underline{u}(x_k+z) \,\dif \mu_v (z)  \\
&= \inf_{v \in \partial B_\epsilon}\left[ \alpha \underline{u}(x_k+v)  + \beta  \int \! \underline{u}(x_k+z) \,\dif \mu_v (z) \right].
\end{align*}

In the following, we denote the vectors chosen by Player I and II as $v_k^\mathrm{I}$ and $v_k^{\mathrm{II}}$, respectively. Irregardless of Player I's choice of strategy $\si$, we have

\begin{align*}
&\E_{\si, \sii^*}^{x_0}[\underline{u}(x_{k+1}) \,|x_0, \dots, x_k] \\
&=    \frac{1 - \delta(x)}{2} \left\{ \left( \alpha \underline{u}(x_k+v_k^\mathrm{I})  + \beta  \int \! \underline{u}(x_k+z) \,\dif \mu_{v_k^\mathrm{I}} (z)\right)  \right.  \\& + \left.
\left(\alpha \underline{u}(x_k+v_k^{\mathrm{II}}) + \beta  \int \! \underline{u}(x_k+z) \,\dif \mu_{v_k^{\mathrm{II}}} (z)\right) \right\} + \delta(x_k)F(x_k)\\
& \leq	 \frac{1 - \delta(x_k)}{2} \left\{ \sup_{v \in  \partial B_\epsilon}\left( \alpha  \underline{u}(x_k+v) + \beta  \int \! \underline{u}(x_k+z) \,\dif \mu_{v} (z)\right)  \right.  \\&
+ \left. \inf_{v \in \partial B_\epsilon}\left(\alpha \underline{u}(x_k+v) + \beta \int \!\underline{u}(x_k+z) \,\dif \mu_{v} (z)  \right) \right\} + \delta(x_k)F(x_k) \\
&= \underline{u}(x_{k}).
\end{align*}

i.e. the stochastic process
\[
M_k = \underline{u}(x_k) 
\]
is a supermartingale. Since by Lemma \ref{gameends} we know that the game ends almost surely, we can now use the optional stopping theorem to deduce that
\[
\begin{split}
u_\II(x_0)&= \inf_{S_{\II}}\sup_{S_{\I}}\,\E_{S_{\I},S_{\II}}^{x_0}[F(x_\tau)] \leq \sup_{S_\I} \E_{S_\I,
S^*_\II}^{x_0}[F(x_\tau)]\\ 
&=\sup_{S_\I}  \E^{x_0}_{S_\I, S^*_\II}[\underline{u} (x_0)]=\underline{u}(x_0),
\end{split}
\]
which proves our claim.
\end{proof}
The following analogous result is proven in the same fashion.
\begin{proclaim}\label{uniq2}
For the function $\underline{u}$ given in Theorem \ref{exis}, we have $ \overbar{u} \leq \ui$.
\end{proclaim}
Combining the inequality (\ref{peliepis}) and Remark \ref{uepis} with Theorems \ref{uniq} and \ref{uniq2}, we have
\[
\uii \leq \underline{u} \leq \overbar{u} \leq \ui \leq \uii,
\]
implying \begin{equation}\label{ovatsamat}\ui = \uii \end{equation} and thus
\begin{proclaim}\label{jatkyk}
There exists a continuous solution to the DPP (\ref{DPP}). 
\end{proclaim}
\section{Measurable uniqueness}
In this section, we prove that there can exist only one measurable function satisfying the DPP with prescribed boundary conditions. As a consequence of this and the previous Theorem \ref{jatkyk}, any measurable function satisfying the DPP must also be continuous. Throughout this section, we assume $\beta \neq 0$.
\begin{proclaim}
Let $u: \Omega_\epsilon \to \R$ be a measurable function satisfying the DPP (\ref{DPP}). Then for all $x \in \Omega_\epsilon$
\[
\inf F \leq u(x) \leq \sup F
\]
\end{proclaim}
\begin{proof}
We consider the latter inequality since the former is proved in a similar fashion.
Assume towards a contradiction that $\sup u > \sup F$. Then for any ${\eta}_1 > 0$ there exists a point $x_1 \in \Omega_\epsilon$ such that
\[
\sup F < u(x_1) 
\] 
and also
\[
u(x_1) > \sup u - {\eta}_1.
\]

We shall first consider the case $x_1 \in \Omega$. For any $x_1 \in \Omega$, we deduce from the DPP that
\begin{align*}
&u(x_1) =   \frac{1}{2}\left[ \sup_{v \in \partial B_\epsilon}\left(\alpha u(x_1+v) + \beta\int \!  u(x_1+z) \,\dif \mu_v (z) \right) \right. \\& + \left. \inf_{v \in \partial B_\epsilon}\left(\alpha u(x_1+v) + \beta  \int \! u(x_1+z) \,\dif \mu_v (z)  \right) \right] \\
&\leq \sup_{v \in \partial B_\epsilon}\left(\alpha u(x_1+v)  + \beta  \int \!  u(x_1+z) \,\dif \mu_v (z)\right)  \\
&\leq \alpha u(x_1) + \sup_{v \in \partial B_\epsilon}\left( \beta \int \!  u(x_1+z) \,\dif \mu_v (z) \right) + \alpha{\eta}_1,
\end{align*}
which can readily be rearranged as
\[
u(x_1) \leq \sup_{v \in \partial B_\epsilon} \int \!  u(x_1+z) \,\dif \mu_v (z) + \frac{\alpha{\eta}_1}{\beta}.
\]
This in turn implies that for any small $\tilde{\eta}_1 > \alpha {\eta}_1 / \beta$ there must exist $v_1 \in \partial B_\epsilon$, for which
\[
u(x_1) - \tilde{\eta}_1 \leq \int \!  u(x_1+z) \,\dif \mu_{v_1} (z).
\]
For our proof, we can choose $\tilde{\eta}_1 = 2 \alpha \eta_1 / \beta$.

We now intend to argue that on the support of the measure $\mu_{v_1}$, we can choose a point $x_2 $ which is at a large enough distance from $x_1$ and at which the function $u$ attains a large enough value. 

Let $0 < \lambda_1 < 1,$ where $\lambda_1\epsilon$ is our desired minimum bound on the distance $||x_2 - x_1||$, to be specified later. We denote 
\[
S_1 := \{ x_2 \in \spt \mu_{v_1} \, | \, ||x_2 - x_1|| \geq \lambda_1 \text{ and } \langle x_1, x_2 - x_1\rangle \geq 0 \}.
\]
We then have
\[
\mu_{v_1}(S_1) = \frac{1 - (\lambda_1)^{n-1}}{2}.
\]
We furthermore denote $\kappa_1 =:  \mu_{v_1}(S_1)$.

Since $u \leq \sup u$, the following is true: On the support of $\mu_{v_1}$ there cannot exist a (measurable) subset $K$ of measure $\kappa_1$, on which 
\begin{equation}\label{kappaat}
u < \sup u - \frac{1 + \alpha}{\beta\kappa_1}\eta_1.
\end{equation}
To prove this, we first estimate
\begin{align*}
 \sup u &\leq u(x_1) + \eta_1  \\ &\leq \int \!  u(x_1+z) \,\dif \mu_{v_1} (z) + \frac{2 \alpha}{\beta}\eta_1+  \eta_1 \\
&= \int_K \!  u(x_1+z) \,\dif \mu_{v_1} (z) + \int_{K^c} \!  u(x_1+z) \,\dif \mu_{v_1} (z) + \frac{1 + \alpha}{\beta}\eta_1\\
&\leq \int_K \!  u(x_1+z) \,\dif \mu_{v_1} (z) + (1-\kappa_1) \sup u + \frac{1 + \alpha}{\beta}\eta_1,
\end{align*}
implying
\[
\sup u - (1-\kappa_1) \sup u \leq \int_K \!  u(x_1+z) \,\dif \mu_{v_1} (z) +\frac{1 + \alpha}{\beta}\eta_1. 
\]
Using (\ref{kappaat}) we further have
\begin{align*}
\kappa_1 \sup u - \frac{1 + \alpha}{\beta}\eta_1   &\leq \int_K \!  u(x_1+z) \,\dif \mu_{v_1} (z) \\ 
& < \kappa_1 \left( \sup u - \frac{1 + \alpha}{\beta\kappa_1}\eta_1 \right) \\
& = \kappa_1 \sup u - \frac{1 + \alpha}{\beta}\eta_1
\end{align*}
and thus (\ref{kappaat}) cannot hold.
Thus, we can choose a point $x_2$ for which \[ || x_1 - x_2 || \geq \lambda_1 \] and 
\[
u(x_2) \geq \sup u - \frac{1 + \alpha}{\beta\kappa_1}\eta_1
\]
We will denote
\[
\eta_2 := \frac{1 + \alpha}{\beta\kappa_1}\eta_1
\]
Repeating the previous steps, we can choose a point $x_3$ for which \[ || x_2 - x_3 || \geq \lambda_2 \] and
\[
u(x_3) \geq \sup u - \frac{1 + \alpha}{\beta\kappa_2}\eta_2
\]
where $0 < \lambda_2 < 1$  can be chosen freely and
\[
\kappa_2 = \frac{1 - (\lambda_2)^{n-1}}{2}.
\]
We now choose $\lambda_j$ such that $\kappa_j = \frac{1}{4}$ for all $j \in \N$. Then, by following the procedure outlined above, we have the following bound
\begin{equation}\label{eetaboundi}
{\eta}_j \leq \left(\frac{4 + 4\alpha}{\beta}\right)^j  {\eta}_1. 
\end{equation}
for all $j \in \N$.
With these choices we can utilize the technique in Lemma \ref{gameends} to establish a lower bound on the distance of point $x_j$ from the origin for all $j \in \N$ as follows: For $j \geq 0$

\begin{align*}
|x_{j+1}|^2 &= |x_j + x_{j+1} - x_j|^2  \\
&= \langle x_j + x_{j+1} - x_j , x_j + x_{j+1} - x_j \rangle \\
&= |x_j|^2 + |x_{j+1} - x_j|^2 + 2 \langle x_j, x_{j+1} - x_j \rangle \\
&\geq |x_j|^2 + \lambda_j.
\end{align*}

Since $\lambda_j$ are bounded uniformly away from zero, this implies that for some $k \in \N$ large enough, $x_k \in I_\epsilon$. Note that the point must pass through $I_\epsilon$ before hitting $O_\epsilon$, since $I_\epsilon$ has width $\epsilon$. Due to the bound given in (\ref{eetaboundi}), if we choose $\eta_1$ small enough in the first step, we can have the value of $u(x_k)$ to be arbitrarily close to $\sup u$, and in particular greater than $\sup F$. We have thus reduced the proof to the case $x_1 \in I_\epsilon$.

Now, from the DPP we deduce
\begin{align*}
&u(x_1) =  \frac{1 - \delta(x_1)}{2} \left[ \sup_{v \in \partial B_\epsilon}\left(\alpha u(x_1+v) + \beta  \int \!  u(x_1+z) \,\dif \mu_v (z) \right) \right. \\& + \left. \inf_{v \in \partial B_\epsilon}\left(\alpha u(x_1+v)  + \beta  \int \! u(x_1+z) \,\dif \mu_v (z) \right) \right] + \delta(x_1)F(x_1) \\
&\leq (1-\delta(x_1))\sup_{v \in \partial B_\epsilon}\left(\alpha  u(x_1+v) + \beta \int \!  u(x_1+z) \,\dif \mu_v (z) \right) + \delta(x_1) \sup F\\
&\leq  (1-\delta(x_1))\sup_{v \in \partial B_\epsilon}\left(\alpha u(x_1) + \beta  \int \!  u(x_1+z) \,\dif \mu_v (z) \right) + \delta(x_1) u(x_1),
\end{align*}
which implies
\[
u(x_1) \leq \sup_{v \in B_\epsilon} \int \!  u(x_1+z) \,\dif \mu_v (z),
\]
as before. Now, if we keep sure that $u(x_j) > \sup F$, we can continue choosing points as before. Eventually, for some $k \in \N$ we will have $x_k \in O_\epsilon$ and $u(x_k) > \sup F$, which is a contradiction, since $u = F$ in $O_\epsilon$.
\end{proof}
As an immediate corollary, we have the following.
\begin{lemma}\label{vikalemma}
For any measurable function $u$ satisfying the DPP, we have
\[
\underline{u} \leq u \leq \overbar{u},
\]
where $\underline{u}$ and $\overbar{u}$ are the lower and upper lower semicontinuous functions obtained in Theorem \ref{exis}.
\end{lemma}
\begin{proof}
By monotonicity of operator $I$ as defined in (\ref{tildei}), we have
\begin{equation}\label{boreley}
\underline{u} = I^n \inf F \leq I^n u \leq I^n \sup F = \overbar{u}.
\end{equation}
Since $u$ satisfies the DPP, $\tilde{I} u = u$. Thus, taking the limit $n \to \infty$ in (\ref{boreley}) gives the claim.  
\end{proof}
Recalling (\ref{ovatsamat}), we thus have
\begin{proclaim}\label{tokavika}
A measurable function $u: \Omega \to \R$ satisfying the DPP with given boundary data is unique.
\end{proclaim}
Finally, since $u$ is thus both lower and upper semicontinuous, we can conclude
\begin{proclaim}
A measurable function satisfying the DPP (\ref{DPP}) is continuous.
\end{proclaim}

\end{document}